\documentclass[10pt]{article}
\usepackage{amsfonts}
\usepackage{amsmath,hhline,epsfig,latexsym}
\usepackage{amssymb}
\usepackage{graphicx}
\usepackage[english]{babel}
\usepackage{hyperref}
\usepackage{xcolor}
\usepackage{ntheorem}
\makeatletter
\renewcommand{\paragraph}{\@startsection{paragraph}{4}{0ex}%
   {-3.25ex plus -1ex minus -0.2ex}%
   {1.5ex plus 0.2ex}%
   {\normalfont\normalsize\bfseries}}
\makeatother

\stepcounter{secnumdepth}
\stepcounter{tocdepth}
\usepackage{amsmath}
\providecommand{\U}[1]{\protect\rule{.1in}{.1in}}
\newtheorem{thm}{Theorem}
\newtheorem{prop}{Proposition}

\newtheorem{rmk}{Remark}

\newtheorem{lem}{Lemma}
\newenvironment{proof}[1][Proof]
{\noindent\textbf{#1:} }{\hfill\rule{0.5em}{0.5em}}

\def \Fin {\hfill\rule{0.5em}{0.5em}}

\newcommand{\R}{\mathbb{R}}

\newcommand{\dis}{\displaystyle}

\title{\textbf{On Some Geometric Inverse Problems for Nonscalar Elliptic Systems}}

\author{
\textsc{Raul K.C. Ara\'ujo}\thanks{Department of Mathematics, Federal University of Pernambuco, UFPE, CEP 50740-545, Recife,
PE, Brazil. E-mail: {\tt  raul@dmat.ufpe.br}. Partially supported by CNPq (Brazil).}
\quad
\and
	\textsc{Enrique Fern\' andez-Cara}\thanks{University of Sevilla, Dpto. E.D.A.N, Aptdo 1160, 41080 Sevilla, Spain. E-mail: {\tt cara@us.es}. Partially supported by grant MTM$2016$-$76990$-P, Ministry of Economy and Competitiveness (Spain).}
	\and
	\textsc{Diego A. Souza}\thanks{Department of Mathematics, Federal University of Pernambuco, UFPE, CEP 50740-545, Recife,
PE, Brazil. E-mail: {\tt diego.souza@dmat.ufpe.br}. Partially supported
by CNPq (Brazil) by the grants $313148$/$2017$-1 and Propesq (UFPE)-
Edital Qualis A. }
}
\date{}
\begin{document}
\maketitle

\begin{abstract}
	In this paper, we consider several geometric inverse problems for linear elliptic systems.
	We prove uniqueness and stability results. In particular, we show the way that the observation depends on the perturbations of the domain.
	In some particular situations, this provides a strategy that could be used to compute approximations to the solution of the inverse problem.
	In the proofs, we use techniques related to (local) Carleman estimates and differentiation with respect to the domain.
\end{abstract}




\section{Introduction} 

	Let $\Omega\subset\mathbb{R}^N$ be a simply connected bounded domain whose boundary $\partial\Omega$ is of class $W^{2,\infty}$, let $D^*$ be a fixed nonempty open set with $D^*\subset\subset\Omega$ and let $\gamma\subset\partial\Omega$ be a nonempty open set. 
	In what follows, the symbols $C, C_1, C_2, \dots$ will be used to denote generic positive constants.
	Sometimes, we will indicate the data on which they depend by writting (for example) $C(\Omega,D^*).$	
	
	Let us consider the following family of subsets of $D^*$:
\[
	\mathcal{D}=\left\{D \subset \Omega:D\text{ is a nonempty simply connected domain,}\ \overline{D}\subset D^*\ \text{and} \ \partial D \text{ is of class } W^{2,\infty}\right\}
\]
	and  let us denote by $\mathcal{A}$ the set of all $(a, b, A, B)$ such that $a,\,b,\, A,\, B\in L^{\infty}(\Omega)$ and
        \begin{equation}\label{ha}
		\left[
		\begin{array}{c}
			\xi_1\\
			\xi_2
		\end{array}
	\right]^t
	\left[
		\begin{array}{cc}
			a(x) & b(x)\\
			A(x) & B(x)\\
		\end{array}
	\right]
	\left[
		\begin{array}{c}
			\xi_1\\
			\xi_2
		\end{array}
	\right]
	\geq -\lambda(|\xi_1|^2 + |\xi_2|^2)\,\,\,\forall (\xi_1,\xi_2)\in\mathbb{R}^2,\,\,\mbox{a.e. in}\,\,\Omega,
        \end{equation}
	 for  some $\lambda$ with $0 < \lambda < \mu_1(\Omega)^{-1}$, where $\mu_1(\Omega)$ is 
	 the smallest positive constant such that
        \begin{equation*}
\displaystyle\|u\|^2_{L^2} \leq \mu_1(\Omega)\displaystyle\|\nabla u\|^2_{L^2}\quad \forall u\in H_0^1(\Omega).
        \end{equation*}

	 In this paper, we will always assume that $(\varphi,\psi)\in H^{1/2}(\partial\Omega)\times H^{1/2}(\partial\Omega)$ and $(a,b,A,B)\in \mathcal{A}$. Under these circumstances it is well known that, for any $D\in \mathcal{D},$ 
	there exists a unique solution $(y,z)\in H^1(\Omega\backslash\overline{D})\times H^1(\Omega\backslash\overline{D})$ to the system
        \begin{equation}\label{an}
\left\{
\begin{array}{lll}
-\Delta y + ay + bz = 0&\mbox{in}& \Omega\backslash\overline{D},\\
-\Delta z + Ay + Bz = 0& \mbox{in}& \Omega\backslash\overline{D},\\ 
y = \varphi,\, z = \psi &\mbox{on}&\partial\Omega,\\
y = z = 0& \mbox{on}&\partial D,
\end{array}
\right.
        \end{equation}
	furthermore satisfying
\[
\|(y,z)\|_{H^1(\Omega\backslash\overline{D})} \leq C(\Omega, D^*)\|(\varphi,\psi)\|_{H^{1/2}(\partial\Omega)}. \label{ha1}
\]


	In many physical phenomena, in order to predict the result of a measurement, we need a model of the system under investigation (typically a PDE system) and an explanation or interpretation of the observed quantities.
	If we are able to compute the solution to the model and quantify relevant observations, we say that we have solved the {\it forward} or {\it direct} problem.
	Contrarily, the {\it inverse} problem consists of using the observations to recover unknown data that characterize the model.
	For details about the main questions concerning inverse problems for PDEs from the theoretical and numerical viewpoints, see for instance the book~\cite{Isakov}.

	In this paper, we will deal with the following inverse geometric problem:
\begin{quote}{\it 
	Given $(\varphi,\psi)\in H^{1/2}(\partial\Omega)\times H^{1/2}(\partial\Omega)$ and $(\alpha,\beta)\in H^{-1/2}(\partial\Omega)\times H^{-1/2}(\partial\Omega)$, find a set $D\in\mathcal{D}$ such that the solution $(y,z)$ to the linear system
	 \eqref{an}
	satisfies the additional conditions:
        \begin{equation}
\left.\frac{\partial y}{\partial n}\right|_{\gamma} = \alpha\,\,\,\mbox{{\it and}}\,\,\,\left.\frac{\partial z}{\partial n}\right|_{\gamma} = \beta. \label{an1}
        \end{equation}}
\end{quote}

	A motivation of problems of this kind can be found, for instance, when one tries to compute the stationary temperature of a chemically reacting plate whose shape is unknown. 
	More precisely, \eqref{an}-\eqref{an1} has the following  interpretation: assume that a chemical product, sensible to temperature effects, fills an unknown domain $\Omega\backslash\overline{D}$; 
	its concentration $y = y(x)$ and its temperature $z = z(x)$ are imposed on the whole outer boundary $\partial\Omega,$ the associated normal fluxes are measured
	on $\gamma\subset \partial \Omega$ and both $y$ and $z$ vanish on the boundary of the non-reacting unknown set $D$; what we pretend to do is to determine $D$ 
	from these data and measurements.

	In the context of the inverse problem \eqref{an}-\eqref{an1}, three main questions appear.
	They are the following:
	
\begin{itemize}

\item{\bf Uniqueness:}
	Let $(\alpha^0,\beta^0)$ and $(\alpha^1,\beta^1)$ be two observations and let $(y^0,z^0)$ and $(y^1,z^1)$ be solutions to \eqref{an} satisfying the identities \eqref{an1} associated to the sets $D^0$ and $D^1$, respectively.  
	The question is: do we have $D^0 = D^1$ whenever $(\alpha^0,\beta^0) = (\alpha^1,\beta^1)$?
	
\item{\bf Stability:}
	Find an estimate of the ``distance''~$\mu_d(D^0,D^1)$ from $D^0$ to $D^1$ in terms of the ``distance'' $\mu_0((\alpha^0,\beta^0),(\alpha^1,\beta^1))$ from~$(\alpha^0,\beta^0)$ to~$(\alpha^1,\beta^1)$ of the form
        \begin{equation*}
\mu_d(D^0,D^1) \leq \Phi(\mu_0((\alpha^0,\beta^0),(\alpha^1,\beta^1))),
        \end{equation*}
where the function $\Phi : \mathbb{R}^+\mapsto \mathbb{R}^+$ satisfies $\Phi(s)\rightarrow 0$ as $s\rightarrow 0$, valid at least whenever $(\alpha^0,\beta^0)$ and $(\alpha^1,\beta^1)$ are ``close'' to a fixed $(\bar{\alpha},\bar{\beta})$.


\item {\bf Reconstruction:}
	Find an iterative algorithm to compute the unknown domain $D$ from the observation~$(\alpha,\beta)$.

\end{itemize}

	In this paper, we focus on the uniqueness and the stability of the inverse problem \eqref{an}-\eqref{an1}. 
	Specifically, our first main result is the following:
	
\begin{thm} \label{an4}
	Assume that $(\varphi,\psi)\in H^{1/2}(\partial\Omega)\times H^{1/2}(\partial\Omega)$ is nonzero.
	For $i=0,1$, let  $(y^i,z^i)$ be the unique weak solution to \eqref{an} with $D$ replaced by~$D^i$ and let $\alpha^i$ and $\beta^i$ be given by the corresponding equalities~\eqref{an1}.
	Then one has the following:
	$$
	(\alpha^0,\beta^0) = (\alpha^1,\beta^1) \quad\Longrightarrow\quad D^0 = D^1.
	$$
\end{thm}

	The proof is given in~Section~\ref{unique}.
	It relies on some ideas from~\cite{Fabre}; more precisely, we use two well known properties of \eqref{an}: unique continuation and well-posedness in the Sobolev space $H^1$.
	
\begin{rmk}{\rm
	Note that, if $(\varphi,\psi)=(0,0),$ then the associated solution to \eqref{an} is zero, for any $D\in \mathcal{D}$.
	Therefore, the uniqueness problem has no sense when $(\varphi,\psi)=(0,0)$. \Fin
	}

\end{rmk}

\begin{rmk}{\rm
	In the one-dimensional case, if one consider \eqref{an} with only one boundary observation, uniqueness may not hold. 
	Indeed, suppose that $\Omega = (0,1)$, $L \in (0,1)$ and consider the system
        \begin{equation}\label{pink 4}
	\left\{
		\begin{array}{lll}
			-y_{xx} + \eta^{2} y + bz = 0&\mbox{in}&(0,L),\\
			-z_{xx} + Ay + \zeta^2 z = 0&\mbox{in}&(0,L),\\
			y(0) = z(0) = 0, &&\\
			y_x(0) = 0,&&
		\end{array}
	\right.
        \end{equation}
where $A,b,\eta,\zeta\in\mathbb{R}$ (all them different from zero) and $|A| +|b| < 2|\eta||\zeta|$.
	Then, $(\eta^2,b,A,\zeta^2)\in\mathcal{A}$ and, using the {\it parameter variation method}, we get that a solution $(y,z)$ to \eqref{pink 4} is given by 
        \begin{equation*}
z(x) = \dfrac{K}{\zeta}\sinh(\zeta x) + 
\dfrac{A}{\zeta}\displaystyle\int_{0}^{x} y(s)\sinh[\zeta(x - s)]\,ds, \quad y(x) = \dfrac{b}{\eta}\displaystyle\int_{0}^{x}z(s)\sinh[\eta(x - s)] ds
        \end{equation*}
 for each $K\in\mathbb{R}$. 
	Therefore, if $K\neq 0$ we have $z_x(0)\neq 0$ and this implies non-uniqueness.
	
	For $N\geq 2$, the uniqueness in the N-dimensional case with only one information on $\gamma$ is, to our knowledge, an open question. \Fin}
\end{rmk}

	In order to state our main stability result, let us introduce some notation.
	Thus, let $D^0 \in \mathcal{D}$ be a fixed subdomain, let $\mu \in W^{1,\infty}(\R^N;\R^N)$ satisfy
	\[
\|\mu\|_{W^{1,\infty}} \leq \epsilon < 1, \ \ \mu = 0 \ \mbox{ in } \ \Omega\backslash\overline{D}^*
	\]
and, for any $\sigma \in (-1,1)$,  let us denote by $m_\sigma$, $D^\sigma$ and~$(y_\sigma,z_\sigma)$ respectively the mapping $m_\sigma := I + \sigma\mu$, the open set $D^\sigma := m_\sigma(D^0)$ and the solution to~\eqref{an} with $D$ replaced by~$D^\sigma$.
	For simplicity, it will be assumed that the coefficients $a,b,A,B$ are constant.
	The following holds:

\begin{thm} \label{stab}
	There exists $\sigma_0 > 0$ with the following properties:
	
\begin{enumerate}

\item The mapping
	        \begin{equation}\label{msigma}
\sigma \mapsto \Bigl.\left(\frac{\partial y_{\sigma}}{\partial n},\frac{\partial z_{\sigma}}{\partial n}\right)\Bigr|_\gamma
	        \end{equation}
is well defined and analytic in~$(-\sigma_0,\sigma_0)$, with values in~$H^{-1/2}(\gamma)^2$.

\item Either $m_\sigma(D^0) = D^0$ for all $\sigma \in (-\sigma_0,\sigma_0)$ (and then the mapping in~\eqref{msigma} is constant), or there exist $\sigma_* \in (0,\sigma_0)$, $C > 0$ and~$k \geq 1$ (an integer) such that
	\[
\left\|\left(\frac{\partial y_{\sigma}}{\partial n},\frac{\partial z_{\sigma}}{\partial n}\right) 
- \left(\frac{\partial y_0}{\partial n},\frac{\partial z_0}{\partial n}\right)\right\|_{H^{-1/2}(\gamma)^2}
\geq C|\sigma|^k \quad \forall \sigma \in (-\sigma_*,\sigma_*).
	\]
\end{enumerate}
\end{thm}

%
%
%

		In~\cite{Kavian} and~\cite{BY}, a similar geometric inverse problem for one scalar elliptic equation is studied.
		For geometric inverse problems for nonlinear models, like Stokes, Navier-Stokes and Boussinesq systems, the uniqueness has been analyzed in~\cite{Alvarez}, \cite{Doubova1} and~\cite{Doubova}, respectively.
		Reconstruction algorithms have been considered and applied in~\cite{Alvarez1} and~\cite{Ben} for the stationary Stokes system and in~\cite{Doubova1} and~\cite{Doubova} for the Navier-Stokes and Boussinesq systems.
		
		Note that, in the applications to fluid mechanics, the goal is to identify the shape of a body around which a fluid flows from measurements performed far from the body.
		In other contexts, the domain $D$ can represent a rigid body immersed in an elastic medium.
		Thus, related inverse problems with relevant applications in Elastography have been analyzed in~\cite{Doubova2} for the wave equation and~\cite{Doubova3} for the Lam\'e system.
		



	This paper is organized as follows.
	In Section~\ref{unique}, we prove a unique continuation property for the solutions to~\eqref{an} and, then, we prove the uniqueness result (Theorem~\ref{an4}).
	Section~\ref{stability} is devoted to proof of the stability result (Theorem~\ref{stab}).
	Finally, in Section~\ref{sec5}, we present  some additional comments and open questions.
	
\section{Unique continuation and uniqueness}\label{unique}

	In this section, we analyze a unique continuation property for \eqref{an}. 
	More precisely, we have the following result:
\begin{thm} \label{mt0}
	Let $G\subset\mathbb{R}^N$ be a bounded domain whose boundary $\partial G$ is of class $W^{1,\infty}$, let $\omega\subset G$ be a nonempty open set 
	and assume that $a, b, A, B\in L^\infty(G)$. Then, any solution $(y,z)\in H^1(G)\times H^1(G)$ to the linear system
        \begin{equation}\label{mt7}
\left\{
\begin{array}{lll}
-\Delta y + ay + bz = 0&\mbox{in}& G,\\
-\Delta z + Ay + Bz = 0&\mbox{in}& G,\\
\end{array}
\right.
        \end{equation}
satisfying
        \begin{equation}
y = z = 0\,\,\mbox{in}\,\, \omega, \label{mt8}
        \end{equation}
 is zero everywhere.
\end{thm}

	As already mentioned, the proof relies on some ideas from \cite{Fabre}. In fact, we will divide the proof in two parts: (i) the proof of Theorem \ref{mt0} when $\omega$ and $G$ are open balls and (ii) the proof for general domains $\omega$ and $G$, 
	using a compactness argument.


\subsection{A unique continuation property for balls}

	In this Section, we prove a very particular result concerning the unique continuation of the solutions to \eqref{mt7}: 
\begin{lem}\label{st1}
	Assume that $R>0$, $x_0\in \mathbb{R}^N$ and $a,b, A, B\in L^\infty(B_{2R}(x_0))$, where $B_{2R}(x_0)$ denotes the open ball of radius $2R$ centered at $x_0$. 
	For any solution $(y,z)\in H^1(B_{2R}(x_0))\times H^1(B_{2R}(x_0))$ to the linear system \eqref{mt7} in $B_{2R}(x_0)$, the following property holds:
\[
	(y,z)=(0,0)~\text{in}~ B_{R}(x_0)\quad \Rightarrow \quad (y,z)=(0,0)~\text{in}~ B_{2R}(x_0).
\]
\end{lem}

	Before proving this lemma, let us introduce $\varphi\in C_0^{\infty}(\mathbb{R}^N)$ and let us define 
$$
	a_0(x,\xi) := \displaystyle|\xi|^2 - |\nabla\varphi(x)|^2 
	\quad \hbox{and} \quad 
	b_0(x,\xi) := 2\xi\cdot\nabla\varphi(x).
$$
	Let us also recall that the {\it Poisson bracket} of $a_0$ and $b_0$ is given by
\[
[a_0,b_0] := \nabla _{\xi} a_0\cdot \nabla_x b_0-\nabla _{x} a_0\cdot \nabla_\xi b_0.
\]
	
	A crucial result in the proof of Lemma \ref{st1} is the following: 
\begin{thm}[{{\cite[Proposition $2.3$]{Fabre}}}]\label{ir}
	Let $U\subset\mathbb{R}^N$ be a nonempty bounded open set, $K\subset U$ a nonempty compact set and assume that $\varphi\in C_0^{\infty}(\mathbb{R}^N)$. 
	Suppose that $\varphi$ is  bi-convex in $U$ with respect to the characteristics of $a_0$ and $b_0$, i.e. $\varphi$ satisfies the following:
        \begin{equation} \label{st3}
	\left\{
		\begin{array}{l}
			\nabla\varphi(x)\neq 0\quad \forall x\in U,\\
			\noalign{\smallskip}\dis
			\exists\,C_0 > 0\,\,\mbox{such that}\,\,~[a_0,b_0](x,\xi) \geq C_0\,\,\mbox{whenever}\,\,(x,\xi)\in U\times\mathbb{R}^N~\mbox{and}~\, a_0(x,\xi) = b_0(x,\xi) = 0.
		\end{array}
	\right.
        \end{equation}
	Then, there exist $C_1 > 0$ and $h_1 > 0$ such that, for all $0 < h < h_1$ and any function $u\in H_0^2(K)$, one has:
        \begin{equation*}
	I_0(u):=\displaystyle\int_{K}e^{2\varphi/h}|u|^2\,dx + h^2\displaystyle\int_{K}e^{2\varphi/h}|\nabla u|^2\,dx \leq C_1h^3\displaystyle\int_{K}e^{2\varphi/h}|\Delta u|^2\,dx.\label{mt}
        \end{equation*}
\end{thm}
\begin{proof} [Proof of Lemma \ref{st1}] 
Without loss of generality we may assume that $x_0=0$. 
Let $(y,z)\in H^1(B_{2R})\times H^1(B_{2R})$ be a solution to \eqref{mt7} such that $y = 0$ and $z = 0$ in $B_R$. 

	We will try to apply Theorem \ref{ir} to the functions $y$ and $z$. To do this, let us fix $\varepsilon > 0$ and let us introduce the sets
\[
K := \left\{x\in\mathbb{R}^N : \frac{3}{4}R\leq |x|\leq 2R - \varepsilon\right\}\quad\hbox{and}\quad U := \left\{x\in\mathbb{R}^N : \frac{1}{2}R< |x|< 2R\right\}  \label{1r2}
\]
	and a function $\varphi\in C_0^{\infty}(\mathbb{R}^N),$ with
        \begin{equation}\label{mt3}
	\varphi(x) := e^{-\delta |x|^2}\,\, \forall x\in\overline{B}_{2R},\ \ \delta > 4/ R^2.
        \end{equation} 
	
	It is not difficult to see that
        \begin{equation}\label{st4}
		\partial_{x_j}\varphi(x) =-2\delta x_j\varphi(x)
		\quad \hbox{and}\quad \partial_{x_j}\partial_{x_k}\varphi(x) = -2\delta\varphi(x)\delta_{jk} + 4\delta^2x_jx_k\varphi(x), 
        \end{equation}
	where the $\delta_{jk}$ are the Kronecker symbols. From \eqref{mt3} and \eqref{st4}, we have that 
        \begin{equation*}
\left.
\begin{alignedat}{2}
[a_0,b_0](x,\xi) 
 = &~64 \, \delta^3\varphi(x)^3|x|^2\left[\delta|x|^2 - 1\right]\\
 		\noalign{\smallskip}\dis
\geq & 16 \, \delta^3R^2e^{-12R^2\delta}\left(\frac{\delta R^2}{4} - 1\right)
\end{alignedat}
\right.
        \end{equation*}
	for any $(x,\xi)\in U\times\mathbb{R}^N$ such that $a_0(x,\xi) = b_0(x,\xi) = 0$.
	Therefore, we see that \eqref{st3} is satisfied by the function $\varphi$ in $U.$

    Let us introduce a cut-off function $\zeta\in C_0^{\infty}(\mathring{K})$ satisfying $\zeta\equiv 1$~for $R - \varepsilon\leq |x|\leq 2R-2\varepsilon$ and let us set
	$\tilde{y} := \zeta y$  and $\tilde{z} := \zeta z$. It is then clear that $(\tilde{y},\tilde{z})\in H_0^2(K)\times H_0^2(K)$.
	After some computations, we obtain that:
\[
	\left\{
		\begin{array}{lll}
			\Delta\tilde{y} = a\tilde{y} + b\tilde{z} + H_1,\\
			\Delta\tilde{z} = A\tilde{y} + B\tilde{z} + H_2,
		\end{array}
	\right.
\]
where 
        \begin{equation} \label{mt1}
H_1 := 2\nabla\zeta\cdot\nabla y + y\Delta\zeta \quad \hbox{and}\quad H_2 := 2\nabla\zeta\cdot\nabla z + z\Delta\zeta.
        \end{equation}

	Consequently, we can apply Theorem \ref{ir} to $\tilde{y}$ and deduce that there exist $C_2> 0$ and $h_2 > 0$ such that
        \begin{equation}\label{ir8}
I_0(\tilde{y}) \leq C_2 h^3 \left( \displaystyle\int_{K}e^{2\varphi/h}|\tilde{z}|^2\,dx + \displaystyle\int_{K}e^{2\varphi/h}|H_1|^2\,dx\right)
        \end{equation}
for all $h\in (0,h_2)$.
	Here, we have absorbed the lower order term for $\tilde y$ from the right hand side by taking $h_2$ small enough.
	Analogously, there exist positive constants $C_3 > 0$ and $h_3 > 0$ such that,
        \begin{equation}\label{ir9}
I_0(\tilde{z}) \leq C_3 h^3 \left(\displaystyle\int_{K}e^{2\varphi/h}|\tilde{y}|^2\,dx
 + \displaystyle\int_{K}e^{2\varphi/h}|H_2|^2\,dx\right)
        \end{equation}
	for all $h\in (0,h_3)$.
	
	Next, adding \eqref{ir8} and \eqref{ir9}, taking $h_4$ sufficiently small and $C_4$ sufficiently large and absorbing again the lower order terms for $\tilde y$ and $\tilde z$ from the right hand side, we have
        \begin{equation}
I_0(\tilde{y}) + I_0(\tilde{z}) \leq C_4 h^3 \displaystyle\int_{K}e^{2\varphi/h}\left(|H_1|^2 + |H_2|^2\right)\,dx, \label{ir11}
        \end{equation}
for all $h\in (0,h_4)$.


	To conclude the proof, we note that \eqref{mt1}, the fact that $y = z = 0$ in $B_R$ and $\nabla\zeta = \Delta\zeta = 0$ for $ R - \varepsilon \leq |x| \leq 2R - 2\varepsilon$ imply that $H_1$ and $H_2$ vanish in $\overline{B}_{2R-2\varepsilon}$. Now, we have from \eqref{mt3} that $\varphi$ is  positive and radially decreasing in $U$. Thus, one has
        \begin{equation} \label{mt4}
\displaystyle\int_{K}e^{2\varphi/h}(|H_1|^2 + |H_2|^2) \,dx \leq e^{2\varphi(2R-2\varepsilon)/h}\displaystyle\int_{K}(|H_1|^2 + |H_2|^2)\,dx.
        \end{equation}
	On the other hand,
        \begin{equation}
	\begin{alignedat}{2}
	\displaystyle\int_{K}e^{2\varphi/h}(|\tilde{y}|^2+|\tilde{z}|^2)\,dx \geq&~ \displaystyle\int_{R\leq |x|\leq 2R - 3\varepsilon}e^{2\varphi/h}(|y|^2+|z|^2)\,dx  \\
	\noalign{\smallskip}\dis
	\geq&~ e^{2\varphi(2R-3\varepsilon)/h}\displaystyle\int_{R\leq |x|\leq 2R - 3\varepsilon}(|y|^2+|z|^2)\,dx. \label{mt5}
	\end{alignedat}
        \end{equation}
It follows from \eqref{ir11}--\eqref{mt5} that
        \begin{equation*}
	 \displaystyle\int_{R\leq |x|\leq 2R - 3\varepsilon}(|y|^2+|z|^2)\,dx \leq C_4 h^3e^{2[\varphi(2R-2\varepsilon)-\varphi(2R-3\varepsilon)]/h}\displaystyle\int_{K}(|H_1|^2 + |H_2|^2)\,dx.
        \end{equation*}
	Since $H_1$ and $H_2$ are independent of $h$ and $\varphi(2R - 2\varepsilon) < \varphi(2R - 3\varepsilon) $, we can let $h\rightarrow 0$ and get that
        \begin{equation*}
y=z = 0\quad \mbox{in}\quad R\leq |x|\leq 2R - 3\varepsilon
        \end{equation*}
and, consequently, $y$ and $z$ vanish in $B_{2R - 3\varepsilon}$. Since $\varepsilon > 0$ is arbitrarily small, we conclude that $y$ and $z$ vanish identically in $B_{2R}$.
\end{proof}

\subsection{Unique continuation for general domains}

	The goal of this section is to prove Theorem \ref{mt0} in the general case.
		
	Let $(y,z)\in H^1(G)\times H^1(G)$ be a solution to \eqref{mt7} satisfying \eqref{mt8} and let us assume that $\overline{B_{\rho_0}(x_0)}\subset\omega$. Let $x_1$ be a point of $G$ and let us see that $y = 0$ and $z = 0$ in a neighborhood of $x_1$.
	
	Since $G$ is connected, there exists a curve $\eta\in C^{\infty}([0,1];G)$ such that  $\eta(0) = x_0$ and $\eta(1) = x_1$. 
	
	Notice that for any $t\in [0,1]$ there exists $r_t > 0$ such that $\overline{B_{2r_t}(\eta(t))}\subset G$. 
	Since $\Gamma:= \eta\left([0,1]\right)$ is a compact set,  there exist~$m\geq 1$ and $0\leq t_1 <\ldots < t_m\leq 1$ satisfying
\[
	\Gamma\subset\bigcup_{j=1}^{m}B_{2r_j}\left(\eta(t_j)\right), \,\mbox{where we have set}\, r_j := r_{t_j}. \label{st}
\]

By construction, setting ~$\rho_1:=\mbox{min}\left\{r_1,\ldots,r_m,\rho_0\right\}$, we have that $\overline{B_{\rho_1}(x)}\subset G$, for all $x\in \Gamma$.

	Finally, we set~$r_0:=\rho_1/2$ and fix $0 < r < r_0$. 
	It is clear that $(y,z)$ vanishes in~$B_r(x_0)$ whence, by Lemma \ref{st1}, $(y,z)$  also vanishes in $B_{2r}(x_0)$. 
	Let $\xi_1: = \eta(\tau_1)\in\partial B_r(x_0)\cap\Gamma$. Then, we have that $(y,z) = (0,0)$ in $B_r(\xi_1)$ and, in view of Lemma \ref{st1}, $(y,z) = (0,0)$ in $B_{2r}(\xi_1)$. 
	Applying the same idea a finite number of times, we obtain $y = 0$ and $z = 0$ in $B_r(x_1).$ This ends the proof.


\subsection{Proof of Theorem 1: uniqueness} 

	Let us introduce the open sets $D:=D^0\cup D^1$ and~$O^0 := \Omega\backslash\overline{D}$ and let $O$ be the unique connected component of $O^0$ 
	such that $\partial\Omega\subset \partial O $. Also, let us set $y := y^0 - y^1$ and $z := z^0 - z^1$ in  $O$.
	Since $\alpha^0 = \alpha^1$ and $\beta^0 = \beta^1$, the couple $(y,z)\in H^1(O)\times H^1(O)$ satisfies:
        \begin{equation}\label{st5}
\left\{
\begin{array}{lll}
-\Delta y + ay + bz = 0&\mbox{in}& O,\\
\noalign{\smallskip}\dis
-\Delta z + Ay + Bz = 0&\mbox{in}& O,\\ 
\noalign{\smallskip}\dis
y =  z = 0&\mbox{on}&\partial \Omega,\\
\noalign{\smallskip}\dis
\frac{\partial y}{\partial n} = \frac{\partial z}{\partial n} = 0&\mbox{on}&\gamma.
\end{array}
\right.
        \end{equation}
	
	Now, we fix $x_0\in\gamma$ and we choose $r > 0$ such that $\overline{B}_r(x_0)\cap\partial \Omega\subset\gamma$.
	Let us set $O' := O\cup B_r(x_0)$ and consider the extension by zero $(\tilde{y},\tilde{z})$ of $(y,z)$ to the whole set $O'$.
	
	From \eqref{st5}, it follows that
        \begin{equation*}
\left\{
\begin{array}{lll}
-\Delta\tilde{y} + a\tilde{y} + b\tilde{z} = 0&\mbox{in}& O',\\
-\Delta\tilde{z} + A\tilde{y} + B\tilde{z} = 0&\mbox{in}& O'.\\
\end{array}
\right.
        \end{equation*}
	Also, since $O'$ is connected and $(\tilde{y},\tilde{z}) = (0,0)$ in $O'\backslash\overline{O}$, Theorem \ref{mt0} implies $(\tilde{y},\tilde{z}) = (0,0)$ in $O'$. 
	In particular, $(y,z) = (0,0)$ in $O$.

	To conclude, let us prove that $D^1\backslash\overline{D}^0$ and $D^0\backslash\overline{D}^1$ must be empty. 
	Thus, let us suppose that $D^1\backslash\overline{D}^0\neq\emptyset$ and let us introduce the set $D^2: = D^1\cup[(\Omega\backslash\overline{D}^0)\cap (\Omega\backslash\overline{O})]$.
	By hypothesis, $D^2\backslash\overline{D}^0$ is nonempty.
	On the other hand, note that $\partial(D^2\backslash\overline{D}^0) := \Gamma_0\cup\Gamma_1$, where $\Gamma_0 = \partial(D^2\backslash\overline{D}^0)\cap\partial D^0$ and $\Gamma_1 = \partial(D^2\backslash\overline{D}^0)\cap\partial D^1$ (see~Figure~1).

	\begin{figure}[h]
\centering
\includegraphics[scale=0.8]{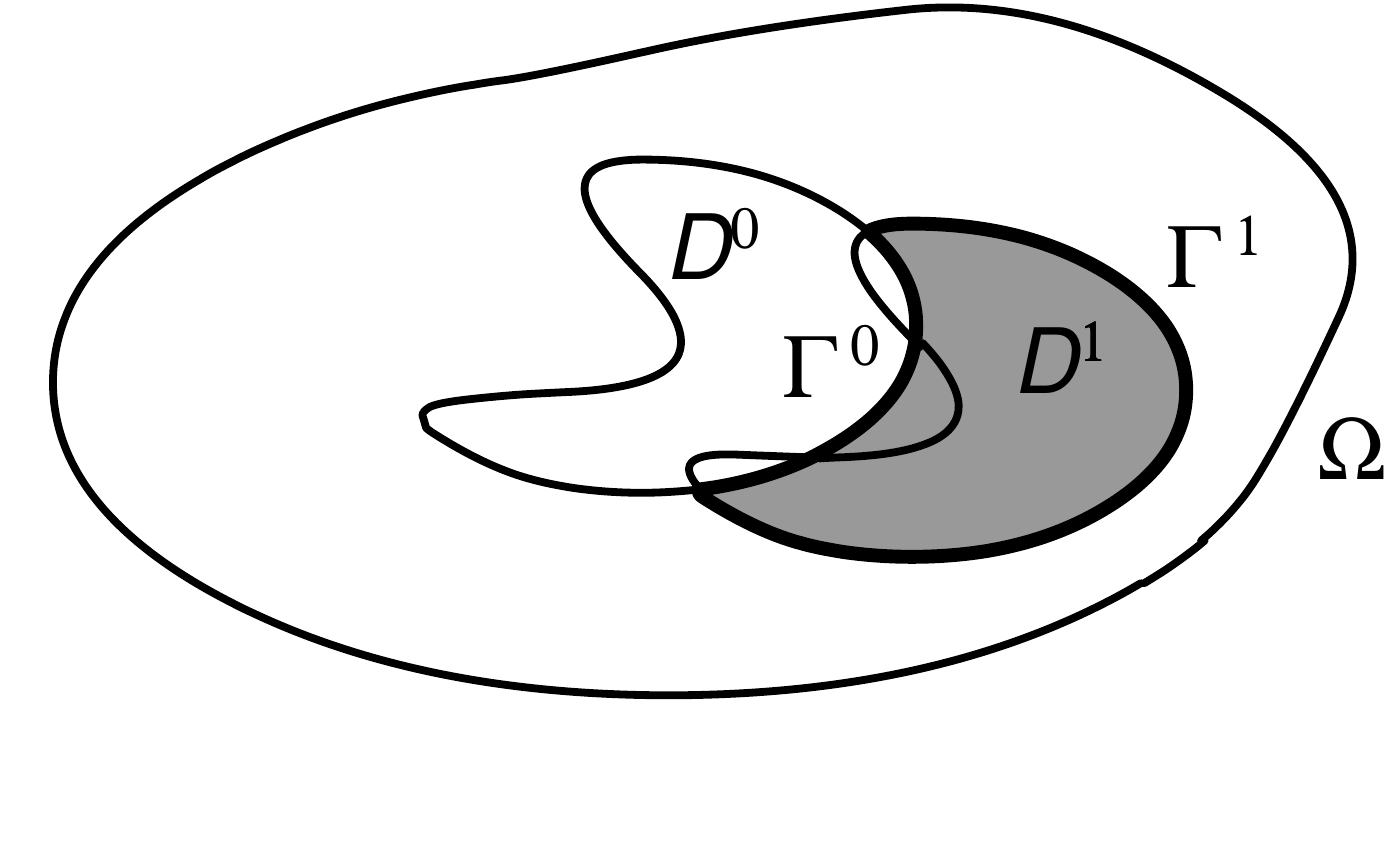}
\vspace{-1cm}
\caption{The filled region is the set $D^2\backslash\overline{D}^0$.}
\end{figure}

	Therefore, since $(y^0,z^0) = (y^1,z^1)$ in $O$, the pair $(y^0,z^0)$ verifies
        \begin{equation}\label{st7}
\left\{
\begin{array}{lll}
-\Delta y^0 + ay^0 + bz^0 = 0&\mbox{in}&D^2\backslash\overline{D}^0,\\
-\Delta z^0 + Ay^0 + Bz^0 = 0&\mbox{in}&D^2\backslash\overline{D}^0,\\ 
y^0 = 0,\, z^0 = 0&\mbox{on}&\Gamma_0,\\
y^0  = 0,\, z^0  = 0&\mbox{on}&\Gamma_1.\\
\end{array}
\right.
        \end{equation}
	Since the linear system \eqref{st7} possesses exactly one solution, we necessarily have $(y^0,z^0) = (0,0)$ in $D^2\backslash\overline{D}^0$. 
	Consequently, in view of Theorem \ref{mt0}, $(y^0,z^0) = (0,0)$ in $\Omega\backslash\overline{D}^0.$ This contradicts the fact that $(\varphi,\psi)$ is not identically zero on $\partial\Omega$.
	Hence, $D^1\backslash\overline{D}^0$ is the empty set.

	Analogously, one can prove that $D^0\backslash\overline{D}^1$ is empty and, finally, one has $D^0 = D^1$.

\section{Stability}\label{stability}

\subsection{Preliminary results}

	Let us introduce some basic notation.
	Let $m = (m^1,\ldots,m^N)\in W^{1,\infty}(\mathbb{R}^N;\mathbb{R}^N)$ be given and let us set
	        \begin{equation*}\label{es}
m':=\left(\dfrac{\partial m^i}{\partial x_j}\right)_{i,j = 1}^{N},\quad \mbox{Jac}(m):=|\det(m')|,
\quad M := ((m')^*)^{-1}. 
	        \end{equation*}

	In the sequel, we will consider the set
	\[
\mathcal{W}_{\epsilon} := \{ \mu \in W^{1,\infty}(\mathbb{R}^N;\mathbb{R}^N) : \|\mu\|_{W^{1,\infty}}\leq \epsilon, \ \ \mu = 0 \ \mbox{ in } \ \Omega\backslash\overline{D}^* \},
	\]
where $0 < \epsilon < 1$.
	We will work with mappings of the form $m := I + \mu$, where $I : \mathbb{R}^N \mapsto \mathbb{R}^N$ is the identity and $\mu \in \mathcal{W}_{\epsilon}$.
	For any $\mu \in \mathcal{W}_{\epsilon}$, $I + \mu$ is obviously bijective, $(I+\mu)^{-1}\in W^{1,\infty}(\mathbb{R}^N;\mathbb{R}^N)$ (see \cite{henrot}, p.~193) and
	\[
(I+\mu)(D) \in \mathcal{D} \quad \forall D \in \mathcal{D}.
	\]
	Also, the corresponding functions $\mbox{Jac}(m)$, $M$ and~$M^{-1}$ satisfy
	        \begin{equation}\label{D3}
\mbox{Jac}(m) \geq C(\epsilon) > 0, \ \ \| M \|_{L^\infty} + \| M^{-1} \|_{L^\infty} \leq C(\epsilon).
	        \end{equation}


	Let $D^0\in\mathcal{D}$ and~$\mu \in\mathcal{W}_{\epsilon}$ be given, let us set again $m := I+\mu$ and~$D^1 = m(D^0)$ and let us consider the solution $(y_1,z_1)\in H^1(\Omega\backslash\overline{D}^1)^2$ to 
	        \begin{equation*}\label{es1}
\left\{
\begin{array}{lll}
-\Delta y_1 + ay_1 + bz_1 = 0&\mbox{in}&\Omega\backslash\overline{D}^1,\\
-\Delta z_1 + Ay_1 + Bz_1 = 0&\mbox{in}&\Omega\backslash\overline{D}^1,\\ 
y_1 = \varphi,\,\, z_1 = \psi &\mbox{on}&\partial\Omega,\\
y_1 = 0,\,\, z_1 = 0&\mbox{on}&\partial D^1.
\end{array}
\right.
	        \end{equation*}

	Since $(\varphi,\psi)\in H^{1/2}(\partial\Omega)^2$, there exists $(\varphi_1,\psi_1)\in H^1(\Omega)^2$ such that
	\[
(\varphi_1,\psi_1) = 0\,\,\mbox{in}\,\,\overline{D^*}\quad\mbox{and}\quad (\varphi_1,\psi_1) = (\varphi,\psi)\,\,\mbox{on}\,\,\partial\Omega.
	\]
	Thus, we can write $(y_1,z_1) = (u_1 + \varphi_1,v_1 + \psi_1)$, where $(u_1,v_1)$ is the solution to 
	        \begin{equation}\label{es2}
\left\{
\begin{array}{lll}
-\Delta u_1 + au_1 + bv_1 = F_1&\mbox{in}&\Omega\backslash\overline{D}^1,\\
-\Delta v_1 + Au_1 + Bv_1 = G_1&\mbox{in}&\Omega\backslash\overline{D}^1,\\ 
u_1 = 0,\,\, v_1 = 0 &\mbox{on}&\partial\Omega\cup\partial D^1
\end{array}
\right.
	        \end{equation}
and
	        \begin{equation*}
F_1 = \Delta\varphi_1 - a\varphi_1 - b\psi_1,\ \ G_1 = \Delta\psi_1 - A\varphi_1 - B\psi_1.
	        \end{equation*}
	We see from \eqref{es2} that for any  $(w,p)\in H_0^1(\Omega\backslash\overline{D}^1)^2$ the following holds:
	        \begin{equation}\label{es3}
\left.
\begin{array}{ll}
\displaystyle\int_{\Omega\backslash\overline{D}^1}(\nabla u_1\cdot\nabla w + \nabla v_1\cdot\nabla p)\,\, dy + \displaystyle\int_{\Omega\backslash\overline{D}^1}(au_1w + bv_1w + Au_1p + Bv_1p)\,\,dy & \\
= - \displaystyle\int_{\Omega\backslash\overline{D}^1}(\nabla \varphi_1\cdot\nabla w + \nabla \psi_1\cdot\nabla p)\,\, dy - \displaystyle\int_{\Omega\backslash\overline{D}^1}(a\varphi_1w + b\psi_1w + A\varphi_1p + B\psi_1p)\,\,dy.&
\end{array}
\right.
	        \end{equation}
	        
	Let us introduce the functions:
	        \begin{equation*}
u_0 := \tilde{m}(u_1),\ \ v_0 := \tilde{m}(u_1), \ \ \varphi_0 := \tilde{m}(\varphi_1),\ \ \psi_0 := \tilde{m}(\psi_1),
	        \end{equation*}
where $\tilde{m}$ is the isomorphism from  $H_0^1(\Omega\backslash\overline{D}^1)^2$ onto~$H_0^1(\Omega\backslash\overline{D}^0)^2$ induced by~$m$, that is,
	        \begin{equation*}
\tilde{m}(f) :=  f\circ m \quad \forall f \in H_0^1(\Omega\backslash\overline{D}^1)^2.
	        \end{equation*}
	Observe that, since $(\varphi_1,\psi_1) = 0$ in~$D^*$ and $m = I$ in~$\Omega\backslash\overline{D}^{*}$, we have $(\varphi_0,\psi_0) = (\varphi_1,\psi_1)$ in~$\Omega$.
	In other words, $(\varphi_1,\psi_1)$ is invariant under the isomorphism $\tilde{m}$ associated to~$m$.

	It can be easily shown that solving the variational problem \eqref{es3} is equivalent to find $(u_0,v_0)\in H_0^1(\Omega\backslash\overline{D}^0)^2$ such that
	        \begin{equation*}
\left.
\begin{array}{ll}
\displaystyle\int_{\Omega\backslash\overline{D}^0}(M\nabla u_0\!\cdot\!M\nabla z \!+\! M\nabla v_0\!\cdot\!M\nabla q)\,\mbox{Jac}(m)\,dx \!+\! \displaystyle\int_{\Omega\backslash\overline{D}^0}(au_0z \!+\! bv_0z \!+\! Au_0q \!+\! Bv_0q)\,\mbox{Jac}(m)\,dx & \\
\ = \!-\! \displaystyle\int_{\Omega\backslash\overline{D}^0}(M\nabla \varphi_0\!\cdot\!\nabla z \!+\! M\nabla \psi_0\!\cdot\!\nabla q)\,\mbox{Jac}(m)\,dx \!-\! \displaystyle\int_{\Omega\backslash\overline{D}^0}(a\varphi_0z \!+\! b\psi_0z \!+\! A\varphi_0q \!+\! B\psi_0q)\,\mbox{Jac}(m)\,dx &
\end{array}
\right.
	        \end{equation*}
for all $(z,q)\in H_0^1(\Omega\backslash\overline{D}^0)$ that is, a solution to the system:
	        \begin{equation}\label{es4}
\left\{
\begin{array}{lll}
-\nabla\cdot (\mbox{Jac}(m)M^*M\nabla u_0) + (au_0 + bv_0)\,\mbox{Jac}(m) = F_0&\mbox{in}&\Omega\backslash\overline{D}^0,\\
-\nabla\cdot  (\mbox{Jac}(m)M^*M\nabla v_0) + (Au_0 + Bv_0)\,\mbox{Jac}(m) = G_0&\mbox{in}&\Omega\backslash\overline{D}^0,\\
u_0 = 0,\,\, v_0 = 0 &\mbox{in} & \partial\Omega\cup\partial D^0,
\end{array}
\right.
	        \end{equation}
where $(F_0,G_0)\in H^{-1}(\Omega\backslash\overline{D}^0)^2$ is given by
	        \begin{equation*}
\left.
\begin{array}{l}
F_0 = \nabla\cdot  (\mbox{Jac}(m)M^*M\nabla \varphi_0) - (a\varphi_0 + b\psi_0)\,\mbox{Jac}(m),\\
G_0 = \nabla\cdot  (\mbox{Jac}(m)M^*M\nabla \psi_0) - (A\varphi_0 + B\psi_0)\,\mbox{Jac}(m).
\end{array}
\right.
	        \end{equation*}

	For convenience, we will rewrite~\eqref{es4} in the abridged form
	        \begin{equation}\label{es5}
\left\{
\begin{array}{lll}
T(u_0,v_0) = (F_0,G_0)&\mbox{in}&\Omega\backslash\overline{D}^0,\\
u_0 = 0,\,\,v_0 = 0 &\mbox{in} &\partial\Omega\cup\partial D^0,
\end{array}
\right.
	        \end{equation}
where the notation is self-explanatory.

\begin{lem}
	The linear operator $T : H_0^1(\Omega\backslash\overline{D}^0)^2 \mapsto H^{-1}(\Omega\backslash\overline{D}^0)^2$ is an isomorphism.
	Furthermore, if $\| \cdot \|_{\mathcal{L}_0}$ denotes the usual norm in $\mathcal{L}(H_0^1(\Omega\backslash\overline{D}^0)^2;H^{-1}(\Omega\backslash\overline{D}^0)^2)$, one has
	\[
\| T \|_{\mathcal{L}_0} + \| T ^{-1} \|_{\mathcal{L}_0} \leq C(\epsilon).
	\]
\end{lem}

\begin{proof}
	It is easy to see that $T\in\mathcal{L}(H_0^1(\Omega\backslash\overline{D}^0)^2;H^{-1}(\Omega\backslash\overline{D}^0)^2)$ and~$\| T \|_{\mathcal{L}_0} \leq C(\epsilon)$.
	On the other hand, the bilinear form $\tau(\cdot\,,\cdot)$, given by
	\[
\tau(\,(u,v),(z,q)\,) := \langle T(u,v), (z,q) \rangle_{H^{-1},H_0^1} \quad \forall (u,v),(z,q)\in H_0^1(\Omega\backslash\overline{D}^0)^2,
	\]
is coercive in view of~\eqref{D3}.
	Indeed, one has
\begin{equation*}
	\left.
	\begin{array}{rl}
	\langle T(u,v), (u,v) \rangle_{H^{-1},H_0^1} \hspace{-0,3cm}
& = \displaystyle\int_{\Omega\backslash\overline{D}^0} (|M\nabla u |^2 + |M\nabla v |^2)\,\mbox{Jac}(m)\,dx \\
 & + \displaystyle\int_{\Omega\backslash\overline{D}^0}(a|u|^2 + B|v|^2 + (b + A)uv)\mbox{Jac}(m)\,dx \\
 & \geq  C(\epsilon)\| (u,v) \|_{H_0^1}
	\end{array}
	\right.
	\end{equation*}
for all $(u,v) \in H_0^1(\Omega\backslash\overline{D}^0)^2$.

	Therefore, from {\it Lax-Milgram's Lemma,} we get the result.
\end{proof}

\begin{thm}\label{es7}
	The mapping $\mu \mapsto (u_0,v_0)$ is analytic in a neighbourhood of the origin in~$\mathcal{W}_{\epsilon}$.
\end{thm}

\begin{proof}
	We note first that $(F_0,G_0)$ does not depend of $\mu$, because $(\varphi_0,\psi_0) = 0$ where $\mu\neq 0$. 
	Then, since the mapping $T$ is a isomorphism we have by \eqref{es5} that 
	        \begin{equation*}\label{es6}
(u_0,v_0) = T^{-1}(F_0,G_0).
	        \end{equation*}
	
	From the results in~\cite{henrot, simon}, we know that the mapping $\mu \mapsto T$ is analytic in a neighbourhood of~$0$.
	Consequently, this is also the case for $\mu \mapsto (u_0,v_0)$ and the proof is done.
\end{proof}

\subsection{Proof of Theorem 2: stability}

%

	Let $D^0\in\mathcal{D}$ and $\mu \in W_{\epsilon}$ be given, with $\mu \neq 0$ in $D^0$.
	Recall that, in~Theorem~2, for any $\sigma \in (-1,1)$, we have set $m_\sigma := I + \sigma \mu $ and~$D^\sigma := m_\sigma(D^0)$ and~$(y_\sigma,z_\sigma)\in H^1(\Omega\backslash \overline{D}_\sigma)^2$ is the solution to
        \begin{equation} \label{pb-Dsigma}
\left\{
\begin{array}{lll}
-\Delta y_\sigma + ay_\sigma + bz_\sigma = 0&\mbox{in}&\Omega\backslash \overline{D}_\sigma,\\
-\Delta z_\sigma + Ay_\sigma + Bz_\sigma = 0&\mbox{in}&\Omega\backslash \overline{D}_\sigma,\\ 
y_\sigma = \varphi,\,\, z_\sigma = \psi &\mbox{on}&\partial\Omega,\\
y_\sigma = 0,\,\, z_\sigma = 0&\mbox{on}&\partial D^\sigma.
\end{array}
\right.
        \end{equation}

	We will argue as in~\cite{Alvarez}:
	
\begin{enumerate}
	
\item	First, it follows from Theorem~\ref{es7} and the fact that $\mu \equiv 0$ in~$\Omega\backslash\overline{D}^*$ that there exists $\sigma_0 > 0$ such that the mapping in~\eqref{msigma} is well defined and analytic in~$(-\sigma_0,\sigma_0)$.
	Hence, there exist $F_1, F_2, \dots$ in~$H^{-1/2}(\gamma)^2$ such that
	\begin{equation} \label{instead}
\left(\frac{\partial y_\sigma}{\partial n},\frac{\partial z_\sigma}{\partial n}\right) - \left(\frac{\partial y_0}{\partial n},\frac{\partial z_0}{\partial n}\right) = \displaystyle\sum_{j=1}^{\infty}\sigma^j F_j \quad \forall \sigma \in (-\sigma_0,\sigma_0),
	\end{equation}
where the series converges in~$H^{-1/2}(\gamma)^2$.

\item Now, let us assume that $m_\sigma(D^0) \not= D^0$ for some $\sigma \in (-\sigma_0,\sigma_0)$.
	In view of~Theorem~1, not all the $F_j$ can be zero.
	Let $k$ be the smallest $j$ such that $F_j \not= 0$.
	It is then clear that there exists $\sigma_* \in (0,\sigma_0)$ such that
	        \begin{equation*}\label{es8}
\left\|\displaystyle\sum_{j = k + 1}^{\infty}\sigma^{j}F_j\right\|_{H^{-1/2}} \leq \frac{1}{2}|\sigma|^{k}\|F_{k}\|_{H^{-1/2}} \quad \forall \sigma \in (-\sigma_*,\sigma_*).
	        \end{equation*}
	Accordingly, for these $\sigma$, one must also have
	\[
|\sigma|^{k} \|F_{k}\|_{H^{-1/2}} \leq 
\left\|\left(\frac{\partial y_{\sigma}}{\partial n},\frac{\partial z_{\sigma}}{\partial n}\right) 
- \left(\frac{\partial y_0}{\partial n},\frac{\partial z_0}{\partial n}\right)\right\|_{H^{-1/2}} 
+ \frac{1}{2} |\sigma|^{k}\|F_{k}\|_{H^{-1/2}} ,
\]
which allows to achieve the proof.

\end{enumerate}

\section{Additional comments and questions}\label{sec5}

\subsection{A similar inverse problem with internal observation}

	Let $\omega\subset\subset\Omega\backslash \overline{D}^*$ be a nonempty open set.
	Consider the following geometric inverse problem, where the observation is performed on $\omega$:

\begin{quote}{\it 
	Given $(\varphi,\psi)\in H^{1/2}(\partial\Omega)\times H^{1/2}(\partial\Omega)$ and $\alpha\in H^1(\omega)$, find a set $D\in\mathcal{D}$ such that the solution $(y,z)$ to the linear system~\eqref{an} satisfies the following additional condition:
        \begin{equation}\label{st9}
	y\big|_{\omega} = \alpha.
        \end{equation}}
\end{quote}
	
	We have the following uniqueness result: 
	
\begin{thm} \label{an5}
	Assume that $(\varphi,\psi)\in H^{1/2}(\partial\Omega)\times H^{1/2}(\partial\Omega)$ is nonzero and suppose that there exists a nonempty open set $\omega_0\subset\omega$ such that $b\neq 0$ a.e.~in $\omega_0$.
	Let~$(y^i,z^i)$ be the unique weak solution to~\eqref{an} with $D$ replaced by~$D^i$ for $i =0,1$ and let $\alpha^i$ be given by the corresponding equality~\eqref{st9}.
	Then, one has:
	$$
		\alpha^0 = \alpha^1 \quad\Longrightarrow\quad D^0 = D^1.
	$$
\end{thm}

\begin{proof}
	The proof is very similar to the proof of Theorem~\ref{an4}.
	
	As before, we can consider the open sets $D:=D^0\cup D^1$, $O^0 := \Omega\backslash\overline{D}$ and the unique connected component $O$ of~$O^0$ such that~$\partial \Omega \subset \partial O$.
	Again, let us set $y := y^0 - y^1$ and $z := z^0 - z^1$ in $O$.
	Then, using the facts that $\alpha^0 = \alpha^1$ and $b \neq 0$ a.e.~in~$\omega_0$, we have that $(y,z)\in H^1(O)\times H^1(O)$ and satisfies
        \begin{equation*}
\left\{
\begin{array}{lll}
	-\Delta y + ay + bz = 0&\mbox{in}& O,\\
	-\Delta z + Ay + Bz = 0&\mbox{in}& O \\
\end{array}
\right.
        \end{equation*}
and
        \begin{equation*}
y = z = 0 \ \mbox{ in } \ \omega_0.
        \end{equation*}

	Consequently, Theorem \ref{mt0} guarantees that $(y, z) = (0,0)$ in $O$. 
	Arguing  as in the proof of Theorem~\ref{an4}, we deduce that $D^0\backslash\overline{D}^1$ and~$D^1\backslash\overline{D}^0$ are empty sets and, consequently, $D^0 = D^1$.
\end{proof}

	We also have a stability result similar to~Theorem~\ref{stab}.
	Thus, let us fix $D^0\in\mathcal{D}$ and $\mu \in W_{\epsilon}$ with $\mu \neq 0$ in $D^0$, let us take $m_\sigma = I + \sigma \mu $ and~$D^\sigma = m_\sigma(D^0)$ and let~$(y_\sigma,z_\sigma)$ be the solution to~\eqref{pb-Dsigma}.
	The following holds:
	
\begin{thm} \label{stab-in}
	Under the assumptions in~Theorem~\ref{an5} on~$(\varphi,\psi)$ and~$b$, there exists $\sigma_0 > 0$ with the following properties:
	
\begin{enumerate}

\item The mapping
	        \begin{equation}\label{Msigma}
\sigma \mapsto y_\sigma\big|_{\omega}
	        \end{equation}
is well defined and analytic in~$(-\sigma_0,\sigma_0)$, with values in~$L^2(\omega)$.

\item Either $m_\sigma(D^0) = D^0$ for all $\sigma \in (-\sigma_0,\sigma_0)$ (and then the mapping in~\eqref{Msigma} is constant), or there exist $\sigma_* \in (0,\sigma_0)$, $C > 0$ and~$k \geq 1$ (an integer) such that
	\[
\left\| (y_\sigma - y_0)\big|_{\omega} \right\|_{L^2}
\geq C|\sigma|^k \quad \forall \sigma \in (-\sigma_*,\sigma_*).
	\]
\end{enumerate}

\end{thm}

\begin{proof}
	Again, the proof is very similar to the proof of stability in the boundary observation case (Theorem~\ref{stab}).
	In fact, the unique difference appears in the last part of the argument, when we write
	\[
(y_\sigma - y_0)\big|_{\omega} = \displaystyle\sum_{j=1}^{\infty}\sigma^j F_j \quad \forall \sigma \in (-\sigma_0,\sigma_0),
	\]
instead of~\eqref{instead}.

	For brevity, we omit the details.
\end{proof}

\begin{rmk}{\rm
	Recall that, in the case of problem $\eqref{an}$--$\eqref{an1}$, we need two boundary observations, the normal derivatives of~$y$ and~$z$ on~$\gamma$, to deduce uniqueness and stability.
	The last two results show that, with internal observations, this holds with the information supplied by just one variable. \Fin}
\end{rmk}


\subsection{A geometric inverse problem for a parabolic system}

	Let us present some ideas that allow to extend Theorems~\ref{an4} and~\ref{an5} to time-dependent parabolic systems.
	For brevity, we will only consider the boundary observation case.
	Thus, let $T > 0$ be given and let us consider the following inverse problem:
	
\begin{quote}
	{\it Given $(\varphi,\psi)$ and $(\alpha,\beta)$ in appropriate spaces and a nonempty open set $\gamma\subset\partial\Omega$, find an open set $D\in\mathcal{D}$ 
	such that the solution $(y,z)$ to the linear evolution system:}
        \begin{equation}\label{an6}
\left\{
\begin{array}{lll}
	 y_t - \Delta y + ay + bz = 0&\mbox{in}&\Omega\backslash\overline{D}\times(0,T),\\
	 z_t - \Delta z + Ay + Bz = 0&\mbox{in}&\Omega\backslash\overline{D}\times(0,T),\\ 
	y = \varphi,\,\, z = \psi &\mbox{on}&\partial\Omega\times(0,T),\\
	y = 0,\,\, z = 0&\mbox{on}&\partial D\times(0,T),\\
	y(\cdot\,,0) = 0,\,\, z(\cdot\,,0) = 0&\mbox{in}&\Omega\backslash\overline{D},
\end{array}
\right.
        \end{equation}
	{\it satisfies the additional conditions:}
        \begin{equation}
	\frac{\partial y}{\partial n}\bigg|_{\gamma\times (0,T)} = \alpha\quad\mbox{and}\quad\frac{\partial z}{\partial n}\bigg|_{\gamma\times (0,T)} = \beta.\label{an7}
        \end{equation}
\end{quote}
	
	If  one assume that $(\varphi,\psi)\not\equiv (0,0)$, then arguments similar to those in the proof of Theorem~\ref{an4} can be used to deduce uniqueness for~\eqref{an6}--\eqref{an7}.
	
	Indeed, the first step is to deduce a unique continuation property:
	
\begin{prop}\label{thm2}
	Let $G\subset\mathbb{R}^N$ be a bounded domain whose boundary is of class~$W^{2,\infty}$ and let us set $Q := G\times(0,T)$. 
	Suppose that $a, b, A, B\in L^{\infty}(Q)$ and let $O$ be a nonempty open subset of $Q$.
	Then, any solution $(y,z)\in L^2(0,T;H^2(G)\times H^2(G))$ to
        \begin{equation}\label{an8}
\left\{
\begin{array}{lll}
 y_t - \Delta y + ay + bz = 0&\mbox{in}& Q,\\ 
 z_t - \Delta z + Ay + Bz = 0&\mbox{in}& Q.
\end{array}
\right.
        \end{equation}
	satisfies the following property:
	\[
(y,z) = (0,0) \ \hbox{ in } \ O \  \Longrightarrow \ (y,z) = (0,0) \ \hbox{ in } \ C(O),
	\]
	where $C(O)$ is the horizontal component of $O$, defined by
        \begin{equation*}
C(O) := \{(x,t)\in Q:\, \exists\, x_0\,\,\mbox{such that}\,\, (x_0,t)\in O\}.
        \end{equation*}
\end{prop}

	The proof of this result is similar to the proof of Theorem~1.4~in~\cite{Fabre}.

\begin{rmk}\label{cor:bound}{\rm
	Let $\Sigma_0$ be an open nonempty subset of $\partial\Omega\times(0,T)$.
	Then, any solution~$(y,z)\in L^2(0,T;H^2(G)\times H^2(G))$ to \eqref{an8} satisfies the following property:
	\[
(y,z) = (0,0) \ \text{ and } \ \left(\frac{\partial y}{\partial n},\frac{\partial z}{\partial n}\right) = (0,0) \ \mbox{ on } \ \Sigma_0 \ \Longrightarrow \ (y,z) = (0,0) \ \mbox{ in } \ C(\Sigma_0).
	\]
	Indeed, take $(x_0,t_0) \in \Sigma_0$.
	Since $\Sigma_0$ is open in $\partial\Omega\times(0,T)$, there exist constants $r,\delta > 0$ such that $\left( B_r(x_0) \cap \partial\Omega \right) \times (t_0 - \delta, t_0 + \delta)\subset\Sigma_0$.
	Let us denote by $(\tilde{y},\tilde{z})$ the extension by zero of~$(y,z)$~to~$O := \left(B(x_0;r)\cap\Omega^c\right)\times(t_0 - \delta, t_0 + \delta)$. 
	Then, $(\tilde{y},\tilde{z})\in L^2(0,T;H^2(G)\times H^2(G))$ is a solution to~\eqref{an8} in~$\left( B(x_0;r)\cup \Omega \right) \times (t_0 - \delta, t_0 + \delta)$ and $(\tilde{y},\tilde{z}) = (0,0)$ in $O$.
	Consequently, by Proposition~\ref{thm2}, $(\tilde{y},\tilde{z})$ vanishes in~$C(O)$.
	Since $(x_0,t_0)$ is arbitrary in~$\Sigma_0$, we get the result. \Fin}
\end{rmk}

	Let us now achieve the proof of uniqueness for the geometric inverse problem \eqref{an6}-\eqref{an7}.
	To this end, let $D^0$ and $D^1$ be two open sets in $\mathcal{D}$ and let $(y^i,z^i)$ be the solution to~\eqref{an6} with~$D = D^i$. 
	Let us also assume that
	\[
\left(\frac{\partial y^0}{\partial n},\frac{\partial z^0}{\partial n}\right) = \left(\frac{\partial y^1}{\partial n},\frac{\partial z^1}{\partial n}\right) \,\,\mbox{on}\,\,\gamma\times(0,T).
	\]
	As before, by introducing the open sets $D:=D^0\cup D^1$, $O^0 := \Omega\backslash\overline{D}$ and $O$, with $y := y^0 - y^1$ and $z := z^0 - z^1$, we have
	\[
\left\{
\begin{array}{lll}
	 y_t - \Delta y + ay + bz = 0&\mbox{in}& O\times(0,T),\\
 	z_t - \Delta z + Ay + Bz = 0&\mbox{in}& O\times(0,T),\\ 
	y = 0,\,\, z = 0&\mbox{on}&\partial \Omega\times(0,T) ,\\
	\dfrac{\partial y}{\partial n} = 0,\,\, \dfrac{\partial z}{\partial n} = 0
	&\mbox{on}&\gamma\times(0,T).
\end{array}
\right.
	\]
	From  Remark~\ref{cor:bound}, we find that $(y,z) = (0,0)$ in~$O\times(0,T)$.
	
	We can prove that $D^1\backslash\overline{D}^0$ is the empty set.
	Indeed, suppose the contrary, i.e.~that  $D^1\backslash\overline{D}^0$ is nonempty.
	Let us introduce the open set $D^2 = D^1 \cup ((\Omega\backslash\overline{D^0} \cap (\Omega\backslash\overline{O}))$.
	As before, using the fact that $(y^0,z^0) = (y^1,z^1)$ in~$O\times(0,T)$, we see that
        \begin{equation}
\left\{
\begin{array}{lll}
 y^0_t - \Delta y^0 + ay^0 + bz^0 = 0&\mbox{in}& (D^2\backslash\overline{D}^0)\times(0,T),\\
 z^0_t - \Delta z^0 + Ay^0 + Bz^0 = 0&\mbox{in}& (D^2\backslash\overline{D}^0)\times(0,T),\\
y^0 = z^0 = 0&\mbox{on}&\partial(D^2\backslash\overline{D}^0)\times (0,T),\\
y^0(\cdot\,,0) = z^0(\cdot\,,0) = 0&\mbox{in}&D^2\backslash\overline{D}^0.\label{an10}
\end{array}
\right.
        \end{equation}
	Consequently, thanks to the uniqueness of solution to~\eqref{an10} and Proposition~\ref{thm2}, we must have $(y^0,z^0) = (0,0)$ in~$(\Omega\backslash\overline{D}^0)\times(0,T)$, which implies $(\varphi,\psi) \equiv (0,0)$, an absurd.
	This proves that $D^1\subset D^0$.
	
	Similarly, we can also prove that $D^0\subset D^1$ and, therefore, $D^0 = D^1$.
	
\begin{rmk}{\rm
	If, in~\eqref{an6}, we impose nonzero initial conditions on~$y$ and/or~$z$, the situation is much more complex.
	In particular, the prevuious argument does not work.
	A detailed analysis will be the objective of a forthcoming paper. \Fin}
\end{rmk}

\begin{rmk}{\rm
		Notice that, in  this time-dependent case,  no assumption of the kind \eqref{ha} is needed. \Fin}
\end{rmk}

	Stability results like Theorems~\ref{stab} and~\ref{stab-in} can also be established in this framework.
	We will not give the details for brevity, since the arguments are not very different and can easily be completed by the reader.


\subsection{Additional comments on stability}

	In the context of the stability problem, we can adopt another (more geometrical) viewpoint.
	To clarify the situation, let us consider the scalar systems
        \begin{equation*}\label{simple-sys}
\left\{
\begin{array}{lll}
	-\Delta y^i = 0 &\mbox{in}&\Omega\backslash\overline{D^i},\\ 
	y  = \varphi^i &\mbox{on}&\partial\Omega, \\
	y  = 0 &\mbox{on}&\partial D^i, \\
\end{array}
\right.
        \end{equation*}
where $\Omega$, is as before, $D^0$ and $D^1$ are convex and have nonempty intersection and the following regularity properties hold:
	\[
\varphi^i \in C^2(\partial\Omega), \ \ \tilde\alpha^i := \frac{\partial y^i}{\partial n} \in C^1(\overline{\gamma}), \ \ y^i \in C^2(\overline{\Omega} \setminus D^i).
	\]
	Let us assume that 
	\[
\left\{
\begin{array}{l} \displaystyle
\|\varphi^i\|_{C^0(\partial\Omega)} \geq m > 0, \ \ \|\varphi^0 - \varphi^1\|_{C^2(\partial\Omega)} \leq \epsilon \\
\noalign{\smallskip}
\|\tilde\alpha^0 - \tilde\alpha^1\|_{C^1(\overline{\gamma})} \leq \epsilon, \ \ 
\|y^i\|_{C^2(\overline{\Omega} \setminus D^i)} \leq M .
\end{array}
\right.
	\]
	Then, it can be proved that the Haussdorf distance $d_H(D^0,D^1)$ satisfies the estimate
	\[
d_H(D^0,D^1) \leq  \frac{C}{\left( \log(\log \frac{1}{\epsilon}) \right)^2} ,
	\]
where $C$ only depends on~$\Omega$, $M$ and~$m$;
	see~\cite{BY}.

	The proof is based on the following well-posedness results, where $C$ is as above:
   
\begin{itemize}

\item Let us set again $G = \Omega \setminus \overline{D^0 \cup D^1}$ and let us assume that
   \[
|y^0(\hat x) - y^1(\hat x)| = \max_{x \in \overline{G}} |y^0(x) - y^1(x)|.
   \]
   Then, under the previous hypotheses, one has 
   $
|y^0(\hat x) - y^1(\hat x)| \leq C \left(\log \frac{1}{\epsilon}\right)^{-1}
   $.

\item Assume that $\|y^0 - y^1\|_{C^0(\overline{G})} \leq \delta$. Then
   $
d_H(D^0,D^1) \leq C\left( \log \frac{1}{\delta} \right)^{-2}
   $.

\end{itemize}

	Under additional properties for the $\varphi^i$, $\tilde\alpha^i$ and $y^i$, the previous estimates can be improved; see~\cite{BY} for more details.

	It would be interesting to extend this approach to the inverse problems~\eqref{an}, \eqref{an1} and~\eqref{an}, \eqref{st9}.
	At present, to our knowledge, whether or not this is possible is an open question.

\subsection{Reconstruction}

	As already said, reconstruction algorithms for the solution of problems of the kind~\eqref{an}--\eqref{an1} have been considered in several papers.
	In all them, the main idea is to reduce to finite dimension and reformulate the search of the unknown $D$ as a constrained (maybe numerically ill-conditionned) extremal problem.
	Then, usual gradient, quasi-Newton or even Newton methods can be used to compute approximate solutions;
	see for instance~\cite{Alvarez, Alvarez1, Ben}.
	
	Let us present in this section a different approach that relies on the domain variation techniques introduced in \cite{Murat,Murat1,simon}.
	
	
	The main idea is to describe how the observation depends on small perturbations of $D$ as explicitly as possible.
	Thus, for each $\mu\in\mathcal{W}_{\epsilon}$, let us set
	\[
D + \mu := \left\{ z\in\mathbb{R}^N z : x + \mu(x), \ \ x\in D\right\}
	\]
and let us recall that, whenever $D\in\mathcal{D}$, we also have $D + \mu\in\mathcal{D}$ as indicated in~Figure~2 below.

\begin{figure}[h]
\centering
\includegraphics[scale = 0.4]{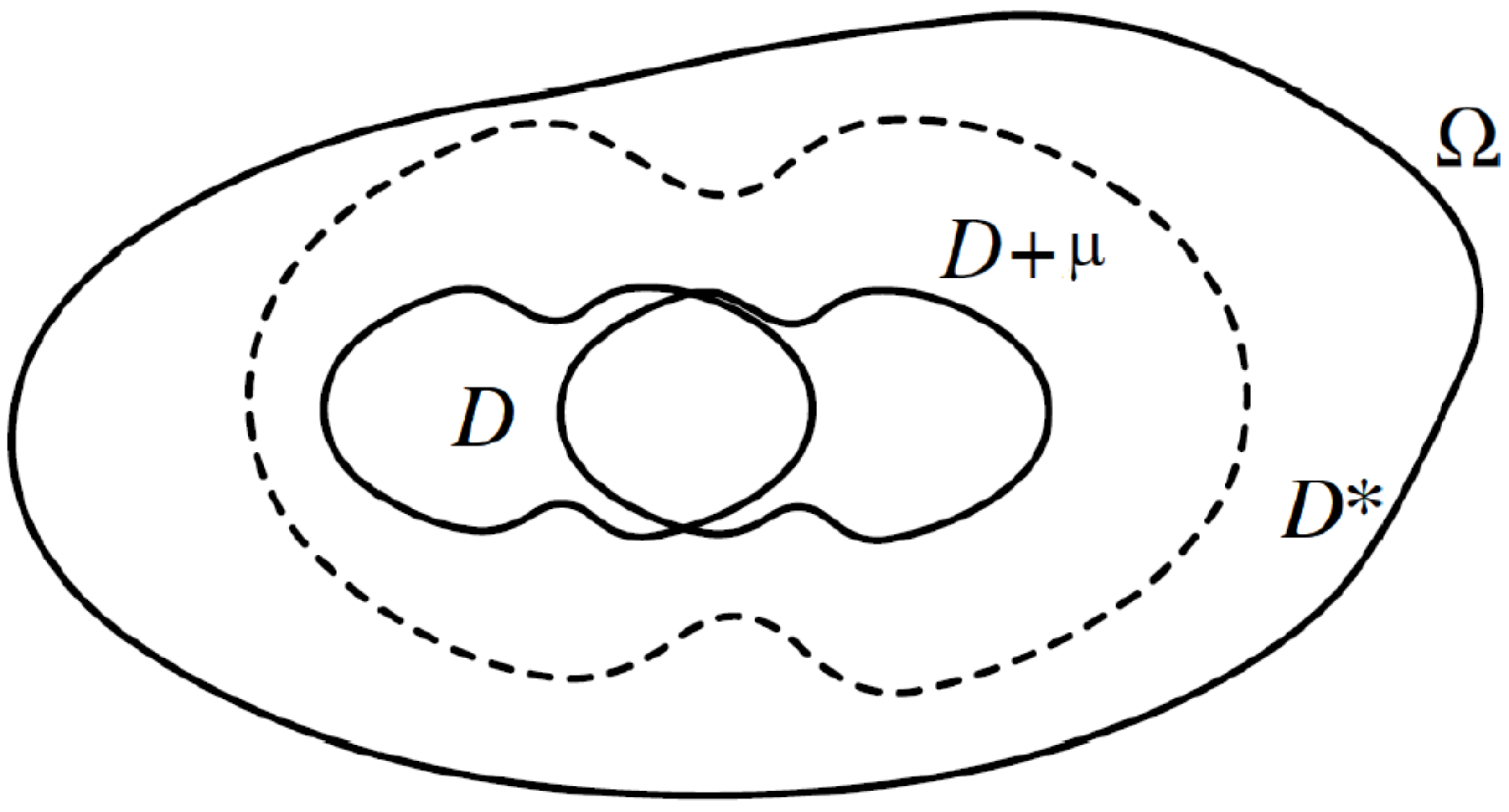}
\caption{The deformations of $D$.}
\end{figure}

	Let us assume that $a,\,b,\,A,\,B\in L^{\infty}(\Omega)$ satisfy \eqref{ha} and
	 we can consider the perturbed system
        \begin{equation*}\label{an11}
\left\{
\begin{array}{lll}
-\Delta y_\mu + ay_\mu + bz_\mu = 0&\mbox{in}&\Omega\backslash(\overline{D+\mu}),\\
-\Delta z_\mu + Ay_\mu + Bz_\mu = 0&\mbox{in}&\Omega\backslash(\overline{D+\mu}),\\ 
y_\mu = \varphi,\,\, z_\mu = \psi &\mbox{on}&\partial\Omega,\\
y_\mu = 0,\,\, z_\mu = 0&\mbox{on}&\partial (D+\mu).
\end{array}
\right.
        \end{equation*}

	Assuming appropriate regularity hypothesis on $D$ and $(\varphi,\psi)$, it can be proved that
        \begin{equation*}\label{an19}
	\left(\frac{\partial y_\mu}{\partial n},\frac{\partial z_\mu}{\partial n}\right) - \left(\frac{\partial y}{\partial n},\frac{\partial z}{\partial n}\right) 
	= \left(\frac{\partial y'_\mu}{\partial n},\frac{\partial z'_\mu}{\partial n}\right) + o(\mu)\quad\mbox{on}\quad \gamma,
        \end{equation*} 
	where $(y'_\mu,z'_\mu)$ is the unique solution to the linear system
        \begin{equation*}
\left\{
\begin{array}{lll}
-\Delta y'_\mu + ay'_\mu + bz'_\mu = 0&\mbox{in}&\Omega\backslash\overline{D},\\
-\Delta z'_\mu + Ay'_\mu + Bz'_\mu = 0&\mbox{in}&\Omega\backslash\overline{D},\\ \label{an15}
y'_\mu = 0,\, z'_\mu = 0 &\mbox{on}&\partial\Omega,\\
y'_\mu = -(\mu\cdot n)\dfrac{\partial y}{\partial n},\, z'_\mu = -(\mu\cdot n)\dfrac{\partial z}{\partial n}&\mbox{on}&\partial D,
\end{array}
\right.
        \end{equation*}
and
	\[
\frac{o(\mu)}{\ \ \,\|\mu\|_{W^{2,\infty}}} \to 0 \ \text{ as } \ \|\mu\|_{W^{2,\infty}} \to 0.
	\]
	Moreover, for any $(\overline{\eta},\overline{\theta})\in C^2(\overline{\gamma})$, one has
        \begin{equation}
	\displaystyle\int_{\gamma}\left[\left(\frac{\partial y_\mu}{\partial n} - \frac{\partial y}{\partial n}\right)\overline{\eta}
	+ \left(\frac{\partial z_\mu}{\partial n} - \frac{\partial z}{\partial n}\right)\overline{\theta}\right]d\Gamma
= -\displaystyle\int_{\partial D}(\mu\cdot n)\left(\frac{\partial y}{\partial n}\frac{\partial\eta}{\partial n} + \frac{\partial z}{\partial n}\frac{\partial\theta}{\partial n}\right)d\Gamma + o(\mu), \label{an17}
        \end{equation}
where $(\eta,\theta)$ is the solution to the adjoint system
        \begin{equation}\label{an18}
\left\{
\begin{array}{lll}
-\Delta \eta + a\eta + A\theta = 0&\mbox{in}& \Omega\backslash\overline{D},\\
\noalign{\smallskip}\dis
-\Delta \theta + b\eta + B\theta = 0& \mbox{in}& \Omega\backslash\overline{D},\\ 
\noalign{\smallskip}\dis
\eta = \overline{\eta}1_{\gamma},\, \theta = \overline{\theta}1_{\gamma} &\mbox{on}&\partial\Omega,\\
\noalign{\smallskip}\dis
\eta = 0,\, \theta = 0& \mbox{on}&\partial D.
\end{array}
\right.
        \end{equation}

	Let us see how, starting from an already computed candidate $\tilde {D}$ to the solution of the geometric inverse problem~\eqref{an}--\eqref{an1}, we can compute a better candidate of the form $\tilde{D} + \mu$.
	
	Let $\mathcal{M}$ be a finite dimensional subspace of $L^\infty(\partial\tilde{D})$ and let $\{f_1,\ldots,f_p\}$ be a basis of $\mathcal{M}$.
	We will take $\mu$ such that $\mu\cdot n|_{\partial\tilde{D}} \in \mathcal{M}$.
	Then, we can write
        \begin{equation*}
\mu\cdot n|_{\partial\tilde{D}} = \displaystyle\sum_{i=1}^{p}\lambda_if_i
        \end{equation*}
for some $\lambda_i\in \mathbb{R}$ to be determined. 

	Now, let us introduce $p$ linearly independent functions $(\overline{\eta}^i,\overline{\theta}^i)\in C^2(\overline{\gamma})$.
	Using \eqref{an17}, we obtain
        \begin{equation*}
\displaystyle\int_{\gamma}\left[\left(\frac{\partial y_\mu}{\partial n} - \tilde{\alpha}\right)\overline{\eta}^{j}1_{\gamma} 
+ \left(\frac{\partial z_\mu}{\partial n} - \tilde{\beta}\right)\overline{\theta}^{j}1_{\gamma}\right]d\Gamma =
-\displaystyle\sum_{i=1}^{p}K_{ij}\lambda_i+ o(\mu),
        \end{equation*}
where 
        \begin{equation*}
K_{ij} := \displaystyle\int_{\partial\tilde{D}}f_i\left(\frac{\partial\tilde{y}}{\partial n}\frac{\partial\eta^{j}}{\partial n} 
+ \frac{\partial\tilde{z}}{\partial n}\frac{\partial\theta^{j}}{\partial n}\right)d\Gamma,
        \end{equation*}
we have denoted by~$(\eta^{j},\theta^{j})$ is the solution to~\eqref{an18} corresponding to~$(\overline{\eta}^{j},\overline{\theta}^{j})$ and~$(\tilde{\alpha}, \tilde{\beta})$ is the observation corresponding to~$(\tilde{y},\tilde{z})$. 
	We thus see that an appropriate strategy to compute the coefficients $\lambda_i$ is to solve, if possible, the finite-dimensional algebraic system
        \begin{equation*}
	\displaystyle\sum_{i=1}^{p}K_{ij}\lambda_i = -\displaystyle\int_{\gamma}\left[(\alpha - \tilde{\alpha})\overline{\eta}^{j}+ (\beta - \tilde{\beta})\overline{\theta}^{j}\right]d\Gamma,
	\quad1\leq j\leq p.
        \end{equation*}

	A rigorous justification of the main steps presented before, together with a detailed analysis of the reconstruction method, will be the subject of a forthcoming work.
	
\

\textbf{Acknowledgments:}
	This work has been partially done while the first author was visiting the University of Sevilla (Spain).
	He wishes to thank the members of the IMUS (Institute of Mathematics of the University of Sevilla) for their kind hospitality.

\end{document}